\documentclass{elsart}
\usepackage{amsfonts}
\usepackage{makeidx}
\usepackage{graphicx}
\usepackage[english,activeacute]{babel}
\usepackage[latin1]{inputenc}
\usepackage{amsmath,amssymb,latexsym,comment}
\usepackage[numbers,sort&compress]{natbib}
\usepackage{hyperref}
\usepackage{fancybox}
\usepackage{fancyhdr}
\usepackage{nonfloat}
\usepackage{multicol}
\usepackage{mathrsfs}

\bibpunct{[}{]}{,}{n}{}{;}
\newtheorem{theorem}{Theorem}[section]

\newtheorem{proposition}{Proposition}[section]

\newcommand{\R}{\mathbb{R}}
\newcommand{\Z}{\mathbb{Z}}

\newenvironment{proof}[1][Proof]{\noindent\textbf{#1.} }{\ \rule{0.5em}{0.5em}}
\journal{J. Differ. Equations Appl.}
\input{tcilatex}
\begin{document}

\begin{frontmatter}



\title{On difference equations of Kravchuk-Sobolev type polynomials of higher order}


\author[AS]{A. Soria-Lorente}
\author[RC]{and Roberto S. Costas-Santos}
%\address[ASL]{Departamento de Matem\'atica - F\'isica,
%	Facultad de Ciencias T\'ecnicas, Universidad de
%	Granma, Km. 17.5 de la carretera de Bayano-Manzanillo, Bayamo, Cuba}
\ead{asorial@udg.co.cu}
\address[AS]{Facultad de Ciencias T\'ecnicas, Universidad de Granma}
\address[RC]{Dpto. de F\'isica y Matem\'{a}ticas, Universidad de Alcal\'{a},
28871, Alcal\'{a} de Henares, Spain}
\ead{rscosa@gmail.com}

%\address{Department of Basic Sciences
%, Granma University\\
%Km 17.5 de la carretera de Bayamo-Manzanillo, Bayamo, Cuba\\
%asorial@udg.co.cu}

\begin{abstract}
In this contribution we consider sequences of monic polynomials orthogonal with
respect to Sobolev-type inner product 
\begin{multline*}
\left\langle f,g\right\rangle _{\lambda,\mu}= \sum_{0 \leq x\leq N}f(x)g(x)\frac{\Gamma(N+1) p^x(1-p)^{N-x} }{\Gamma (N-x+1) \Gamma(x+1) }\\+\lambda\Delta^j f(0)\Delta^j g(0)+\mu\Delta^j f(N)\Delta^j g(N),
\end{multline*}%
where $0<p <1$, $\lambda,\mu\in \R_{+}$, $n\leq N\in \Z_{+}$, $j\in \Z_{+}$ and $\Delta$ denotes the forward difference operators. We derive an explicit representation for these polynomials. In addition, the ladder operators associated with these polynomials are obtained. As a consequence, the linear difference equations of second order are also given.

\end{abstract}

\end{frontmatter}

\section{Introduction}

\label{[S1]-Intro}

%%%%%%%%%%%%%%%%%%%%%%%%%%%%%%%%%%%%%%%%%%%%%%%%%%%%%%%%%%%%%%%%%%%%%%%%%%%%%%%%%%%%%%%%%%%%%%%%%%%%%%%

%%%%%%%%%%%%%%%%%%%%%%%%%%%%%%%%%%%%%%%%%%%%%%%%%%%%%%%%%%%%%%%%%%%%%%%%%%%%%%%%%%%%%%%%%%%%%%%%%%%%%%%

The study of the sequences of polynomials orthogonal with respect to the
Sobolev inner product %\begin{equation}

In this contribution we will focus our attention on the sequence $\{\mathbb{K%
}_{n}^{(j)}\}_{n\geq 0}$ of monic polynomials orthogonal with respect to the
following inner product on $\mathbb{P}$ involving differences of higher
order 
\begin{multline}
\left\langle f,g\right\rangle _{\lambda ,\mu }=\sum_{0\leq x\leq N}f(x)g(x)%
\frac{\Gamma (N+1)p^{x}(1-p)^{N-x}}{\Gamma (N-x+1)\Gamma (x+1)} \\
+\lambda \Delta ^{j}f(0)\Delta g(0)+\mu \Delta ^{j}f(N)\Delta g(N),
\label{SobIP}
\end{multline}%
where $0<p <1$, $\lambda,\mu\in \R_{+}$, $n\leq N\in \Z_{+}$, $j\in \Z_{+}$ and $\Delta $ denotes the forward difference operators defined by $%
\Delta f\left( x\right) =f\left( x+1\right) -f\left( x\right) $, where $%
\Delta ^{j}f(x)=\Delta \lbrack \Delta ^{j-1}f(x)]$. In this work we get the
connection formula and the hypergeometric representation of the polynomials
orthogonal with respect to \eqref{SobIP}. Moreover, we find the ladder
(creation and annihilation) operators associated with such a sequence of
polynomials. Finally, we deduce the second order linear difference equation
that those Sobolev type polynomials satisfy.

The structure of the paper is the following: In Section 2, we introduce some
preliminary results about Kravchuk polynomials which will be very useful in
the analysis presented. In Section 3, we obtain the connection formula
between the Kravchuk polynomials and the polynomials orthogonal with respect
to \eqref{SobIP}, as well as we deduce the hypergeometric representation of
such polynomials. Finally, in Section 4, we find the ladder (creation and
annihilation) operators for the sequence of orthogonal polynomials of
Sobolev type. As a consequence, the second order linear difference equations
associated with them are deduced.

\section{Preliminary results}

\label{[S2]-Prelim}

Let $\{K_{n}^{p,N}\}_{n\geq 0}$ be the sequence of monic Kravchuk
polynomials \cite{alv2}, orthogonal with respect to the inner
product on $\mathbb{P}$ 
\begin{equation*}
\left\langle f,g\right\rangle =\sum_{0 \leq x\leq N}f(x)g(x)\frac{%
\Gamma(N+1) p^x(1-p)^{N-x} }{\Gamma (N-x+1) \Gamma(x+1) },\quad 0<p <1,
\end{equation*}%
which can be explicitly given by%
\begin{equation}
K_{n}^{p,N}\left( x\right) =p^n\left(-N\right) _{n}\,_{2}F_{1}\left( 
\begin{array}{c|c}
-n,-x &  \\ 
& p^{-1} \\ 
-N & 
\end{array}%
\right) ,  \label{MRH}
\end{equation}%
with $n\leq N$, please refer to \cite%
{algama,alv2}. Here, $_{r}F_{s}$
denotes the ordinary hypergeometric series defined by%
\begin{equation}
_{r}F_{s}\left( 
\begin{array}{c|c}
a_{1},\ldots ,a_{r} &  \\ 
& x \\ 
b_{1},\ldots ,b_{s} & 
\end{array}%
\right) =\sum_{k\geq 0}\frac{\left( a_{1},\ldots ,a_{r}\right) _{k}}{\left(
b_{1},\ldots ,b_{s}\right) _{k}}\frac{x^{k}}{k!},  \label{HS}
\end{equation}%
where 
\begin{equation*}
\left( a_{1},\ldots ,a_{r}\right) _{k}:=\prod_{1\leq i\leq r}\left(
a_{i}\right) _{k},
\end{equation*}%
and $\left( \cdot \right) _{n}$ denotes the Pochhammer symbol \cite{arso, gara}, also called as the shifted factorial, defined by 
\begin{equation*}
(x)_{n}=\prod_{0\leq j\leq n-1}\left( x+j\right) ,\quad n\geq 1,\quad \left(
x\right) _{0}=1.
\end{equation*}%
Moreover, $\left\{a_{i}\right\} _{i=1}^{r}$ and $\left\{ b_{j}\right\}
_{j=1}^{s}$ are complex numbers subject to the condition that $b_{j}\neq -n$
with $n\in \mathbb{N}\backslash \left\{ 0\right\} $ for $j=1,2,\ldots ,s$.

Next, we summarize some basic properties of Meixner orthogonal polynomials
to be used in the sequel.

\begin{proposition}
\label{S2-PROP-1}Let $\{K_{n}^{p,N}\}_{n\geq 0}$ be the classical Kravchuk
sequence monic orthogonal polynomials. The following statements hold.

\begin{enumerate}
\item Three term recurrence relation \cite{arcovan} 
\begin{equation}
xK_{n}^{p,N}\left( x\right) =K_{n+1}^{p,N}\left( x\right) +\alpha
_{n}^{p,N}K_{n}^{p,N}\left( x\right) +\beta _{n}^{p,N}K_{n-1}^{p,N}\left(
x\right) ,\quad n\geq 0,  \label{ReR}
\end{equation}%
where%
\begin{equation*}
\alpha _{n}^{p,N}=p(N-n)+n(1-p)\quad\text{and}\quad \beta
_{n}^{p,N}=np(1-p)(N-n+1),
\end{equation*}%
with initial conditions $K_{-1}^{p,N}\left( x\right) =0$, and $%
K_{0}^{p,N}\left( x\right) =1$.

\item Structure relations. For every $n\in \mathbb{N}$, \cite%
{alv2} 
\begin{equation}
x\nabla K_{n}^{p,N}\left( x\right) =nK_{n}^{p,N}\left( x\right)
+np(N-n+1)K_{n-1}^{p,N}\left(x\right).  \label{StruR1}
\end{equation}

\item Squared norm. For every $n\in \mathbb{N}$, \cite{alv2} 
\begin{equation}
\left\Vert K_{n}^{p,N}\right\Vert ^{2}=n!(-N)_np^n(p-1)^n.  \label{Norm2}
\end{equation}

\item Orthogonality relation. For $a<0$ 
\begin{equation*}
\sum_{0 \leq x\leq N}K_{m}^{p,N}(x)K_{n}^{p,N}(x)\frac{\Gamma(N+1)
p^x(1-p)^{N-x} }{\Gamma (N-x+1) \Gamma(x+1) }=\left\Vert
K_{n}^{p,N}\right\Vert ^{2}\delta _{m,n},
\end{equation*}%
where by $\delta _{i,j}$ we denote the Kronecker delta function.

\item Second order difference equation (hypergeometric type equation) 
\begin{equation*}
(1-p)x\Delta \nabla K_{n}^{p,N}(x) + (Np-x)\Delta K_{n}^{p,N}(x) +
nK_{n}^{p,N}(x) =0.
\end{equation*}

\item Forward operator 
\begin{equation*}
\Delta ^{k}K_{n}^{p,N}(x) =\left[ n\right] _{k}K_{n-k}^{p,N-k}(x),
\end{equation*}%
where $\left[ \cdot \right] _{n}$ denotes the Falling Factorial \cite[p. 6]%
{arso}, defined by%
\begin{equation*}
\left[ z\right] _{n}\equiv \left( -1\right) ^{n}\left( -z\right) _{n},\quad
n\geq 1,\quad \left[ z\right] _{0}=1.
\end{equation*}
\end{enumerate}
\end{proposition}

Furthermore, we denote the $n$-th reproducing kernel by%
\begin{equation*}
{\mathscr K}_{n}(x,y)=\sum_{0\leq k\leq n}\frac{K_{k}^{p,N}(x)K_{k}^{p,N}(y)%
}{\left\Vert K_{k}^{p,N}\right\Vert ^{2}}.  \label{Kernel1}
\end{equation*}%
Then, for all $n\in \mathbb{N}$,%
\begin{equation*}
{\mathscr K}_{n}(x,y)=\frac{1}{\left\Vert K_{n}^{p,N}\right\Vert ^{2}}\frac{%
K_{n+1}^{p,N}(x)K_{n}^{p,N}(y)-K_{n+1}^{p,N}(y)K_{n}^{p,N}(x)}{x-y}.
\label{CDarb}
\end{equation*}%
Provided $\Delta^{k}f(x)=\Delta\left(\Delta^{k-1}f(x)\right)$ for the
partial finite difference of ${\mathscr K}_{n}(x,y)$ we will use the
following notation 
\begin{equation}
{\mathscr K}_{n}^{(i,j)}(x,y)=\Delta_{x}^{i}\left( \Delta_{y}^{j}{\mathscr K}%
_{n}\left( x,y\right) \right) =\sum_{0\leq k\leq n}\frac{%
\Delta^{i}K_{k}^{p,N}(x)\Delta^{j}K_{k}^{p,N}(y)}{\left\Vert
K_{k}^{p,N}\right\Vert ^{2}}.  \label{Kij}
\end{equation}%
The following result is obtained of similar way to \cite[Proposition 2.2]%
{arsoII}.

\begin{proposition}
Let $\left\{ K_{n}^{p,N}\right\} _{n\geq 0}$ be the sequence of monic
Kravchuk orthogonal polynomials. Then, the following statement holds, for
all $n\in \mathbb{N}$, 
\begin{equation}
{\mathscr K}_{n-1}^{\left( 0,j\right) }\left( x,y\right) = {\mathscr A}%
_n^{(j)}(x,y)K_{n}^{p,N}\left( x\right)+{\mathscr B}%
_n^{(j)}(x,y)K_{n-1}^{p,N}\left( x\right),  \label{Kernel0j}
\end{equation}
where 
\begin{equation*}
{\mathscr A}_n^{(j)}(x,y) = \frac{j!}{\left\Vert K_{n-1}^{p,N}\right\Vert
^{2}[x-y]_{j+1}}\sum_{0\leq k\leq j}\frac{\Delta^{k}K_{n-1}^{p,N}\left(
y\right) }{k!}[x-y]_{k},
\end{equation*}
and 
\begin{equation*}
{\mathscr B}_n^{(j)}(x,y) = -\frac{j!}{\left\Vert K_{n-1}^{p,N}\right\Vert
^{2}[x-y]_{j+1}}\sum_{0\leq k\leq j}\frac{\Delta^{k}K_{n}^{p,N}\left(
y\right) }{k!}[x-y]_{k}.
\end{equation*}
\end{proposition}

\section{Connection formula and hypergeometric representation of $\mathbb{K}%
_{n}^{(j)}(x)$}

In this section, we first express the Kravchuk-Sobolev type orthogonal
polynomials $\mathbb{K}_{n}^{(j)}(x)$ in terms of the monic Kravchuk
orthogonal polynomials $K_{n}^{p,N}\left( x\right) $ and the Kernel
polynomial $\left( \text{\ref{Kij}}\right) $. Taking into account the
Fourier expansion we have 
\begin{equation*}
\mathbb{K}_{n}^{(j)}(x)=K_{n}^{p,N}\left( x\right) +\sum_{0\leq k\leq
n-1}a_{n,k}K_{k}^{p,N}\left( x\right) .
\end{equation*}%
Then, from the properties of orthogonality of $K_{n}^{p,N}$ and $\mathbb{K}%
_{n}^{(j)}$ respectively, we arrived to 
\begin{equation*}
a_{n,k}=-\frac{\lambda \Delta ^{j}\mathbb{K}_{n}^{(j)}(0)\Delta
^{j}K_{k}^{p,N}(0)+\mu \Delta ^{j}\mathbb{K}_{n}^{(j)}(N)\Delta
^{j}K_{k}^{p,N}(N)}{\left\Vert K_{k}^{p,N}\right\Vert ^{2}},
\end{equation*}%
for $0\leq k\leq n-1$. Thus, we get 
\begin{multline*}
\mathbb{K}_{n}^{(j)}(x)=K_{n}^{p,N}\left( x\right) -\lambda \Delta ^{j}%
\mathbb{K}_{n}^{(j)}\left( 0\right) {\mathscr K}_{n-1}^{\left( 0,j\right)
}\left( x,0\right) \\
-\mu \Delta ^{j}\mathbb{K}_{n}^{(j)}\left( N\right) {\mathscr K}%
_{n-1}^{\left( 0,j\right) }\left( x,N\right) .
\end{multline*}%
After some manipulations, we deduce the following linear system with two
unknowns, namely $\Delta ^{j}\mathbb{K}_{n}^{(j)}\left( 0\right) $ and $%
\Delta ^{j}\mathbb{K}_{n}^{(j)}\left( N\right) $, 
\begin{equation*}
\begin{pmatrix}
1+\lambda {\mathscr K}_{n-1}^{\left( j,j\right) }\left( 0,0\right) & \mu {%
\mathscr K}_{n-1}^{\left( j,j\right) }\left( 0,N\right) \\ 
\lambda {\mathscr K}_{n-1}^{\left( j,j\right) }\left( N,0\right) & 1+\mu {%
\mathscr K}_{n-1}^{\left( j,j\right) }\left( N,N\right)%
\end{pmatrix}%
\cdot 
\begin{pmatrix}
\Delta ^{j}\mathbb{K}_{n}^{(j)}\left( 0\right) \\ 
\Delta ^{j}\mathbb{K}_{n}^{(j)}\left( N\right)%
\end{pmatrix}%
=%
\begin{pmatrix}
\Delta ^{j}K_{n}^{p,N}\left( 0\right) \\ 
\Delta ^{j}K_{n}^{p,N}\left( N\right)%
\end{pmatrix}%
.
\end{equation*}%
Thus, using Cramer's rule we obtain 
\begin{equation*}
\mathbb{K}_{n}^{(j)}(x)=K_{n}^{p,N}\left( x\right) -\lambda {\mathscr K}%
_{n-1}^{\left( 0,j\right) }\left( x,0\right) \Phi _{1}(j,n)-\mu {\mathscr K}%
_{n-1}^{\left( 0,j\right) }\left( x,N\right) \Phi _{2}(j,n),
\end{equation*}%
where 
\begin{equation*}
\Phi _{1}(j,n)=\frac{%
\begin{vmatrix}
\Delta ^{j}K_{n}^{p,N}\left( 0\right) & \mu {\mathscr K}_{n-1}^{\left(
j,j\right) }\left( 0,N\right) \\ 
\Delta ^{j}K_{n}^{p,N}\left( N\right) & 1+\mu {\mathscr K}_{n-1}^{\left(
j,j\right) }\left( N,N\right)%
\end{vmatrix}%
}{%
\begin{vmatrix}
1+\lambda {\mathscr K}_{n-1}^{\left( j,j\right) }\left( 0,0\right) & \mu {%
\mathscr K}_{n-1}^{\left( j,j\right) }\left( 0,N\right) \\ 
\lambda {\mathscr K}_{n-1}^{\left( j,j\right) }\left( N,0\right) & 1+\mu {%
\mathscr K}_{n-1}^{\left( j,j\right) }\left( N,N\right)%
\end{vmatrix}%
},
\end{equation*}%
and 
\begin{equation*}
\Phi _{2}(j,n)=\frac{%
\begin{vmatrix}
1+\lambda {\mathscr K}_{n-1}^{\left( j,j\right) }\left( 0,0\right) & \Delta
^{j}K_{n}^{p,N}\left( 0\right) \\ 
\lambda {\mathscr K}_{n-1}^{\left( j,j\right) }\left( N,0\right) & \Delta
^{j}K_{n}^{p,N}\left( N\right)%
\end{vmatrix}%
}{%
\begin{vmatrix}
1+\lambda {\mathscr K}_{n-1}^{\left( j,j\right) }\left( 0,0\right) & \mu {%
\mathscr K}_{n-1}^{\left( j,j\right) }\left( 0,N\right) \\ 
\lambda {\mathscr K}_{n-1}^{\left( j,j\right) }\left( N,0\right) & 1+\mu {%
\mathscr K}_{n-1}^{\left( j,j\right) }\left( N,N\right)%
\end{vmatrix}%
}.
\end{equation*}%
Then, from \eqref{Kernel0j} and the previous we deduce 
\begin{equation}
\mathbb{K}_{n}^{(j)}(x)={\mathscr C}_{1,n}^{(j)}(x)K_{n}^{p,N}\left(
x\right) +{\mathscr D}_{1,n}^{(j)}(x)K_{n-1}^{p,N}\left( x\right) ,
\label{SobP1}
\end{equation}%
where 
\begin{equation*}
{\mathscr C}_{1,n}^{(j)}(x)=1-\lambda \Phi _{1}(j,n){\mathscr A}%
_{n}^{(j)}(x,0)-\mu \Phi _{2}(j,n){\mathscr A}_{n}^{(j)}(x,N),
\end{equation*}%
and 
\begin{equation*}
{\mathscr D}_{1,n}^{(j)}(x)=-\lambda \Phi _{1}(j,n){\mathscr B}%
_{n}^{(j)}(x,0)-\mu \Phi _{2}(j,n){\mathscr B}_{n}^{(j)}(x,N).
\end{equation*}%
Next we will focus our attention in the representation of $Q_{n}^{i,\lambda
}\left( x\right) $ as hypergeometric functions.

\begin{theorem}
The monic Kravchuk-Sobolev orthogonal polynomials $\mathbb{K}_{n}^{(j)}(x)$
have the following hypergeometric representation for $n\leq N\in \Z_{+}$, $j\in \Z_{+}$,
\begin{equation}
\mathbb{K}_{n}^{(j)}(x)=p^{n-1}\left( -N\right) _{n-1}h_{n}^{\left( j\right)
}\left( x\right) \,_{3}F_{2}\left( 
\begin{array}{c|c}
-n,-x,f_{n}^{\left( j\right) }\left( x\right) &  \\ 
& p^{-1} \\ 
-N,f_{n}^{\left( j\right) }\left( x\right) -1 & 
\end{array}%
\right) ,  \label{MSPHR}
\end{equation}%
where $f_{n}^{\left( j\right) }\left( x\right) $ is given in $\left( \text{%
\ref{fx}}\right) $ and%
\begin{equation*}
h_{n}^{\left( j\right) }\left( x\right) =-\left( p\left( N-n+1\right) {%
\mathscr C}_{n}^{(j)}(x)-{\mathscr D}_{n}^{(j)}(x)\right) .
\end{equation*}
\end{theorem}

\begin{proof}
In fact, having into account%
\begin{equation*}
\left( -x\right) _{k}=0,\quad \mbox{if}\quad x<k.
\end{equation*}%
as well as $\left( \text{\ref{MRH}}\right) $ and $\left( \text{\ref{SobP1}}%
\right) $ we deduce%
\begin{multline*}
\mathbb{K}_{n}^{(j)}(x)=p^{n}\left( -N\right) _{n}\,{\mathscr C}%
_{1,n}^{(j)}(x)\sum_{0\leq k\leq n}\frac{\left( -n\right) _{k}\left(
-x\right) _{k}}{\left( -N\right) _{k}}\frac{\left( p^{-1}\right) ^{k}}{k!} \\
+p^{n-1}\left( -N\right) _{n-1}\,{\mathscr D}_{1,n}^{(j)}(x)\sum_{0\leq
k\leq n-1}\frac{\left( 1-n\right) _{k}\left( -x\right) _{k}}{\left(
-N\right) _{k}}\frac{\left( p^{-1}\right) ^{k}}{k!}.
\end{multline*}%
Then, using the identity%
\begin{equation*}
\frac{a+k}{a}=\frac{\left( a+1\right) _{k}}{\left( a\right) _{k}},
\end{equation*}%
we get%
\begin{multline*}
\mathbb{K}_{n}^{(j)}(x)=p^{n}\left( -N\right) _{n}\,{\mathscr C}%
_{1,n}^{(j)}(x)\sum_{0\leq k\leq n}\frac{\left( -n\right) _{k}\left(
-x\right) _{k}}{\left( -N\right) _{k}}\frac{\left( p^{-1}\right) ^{k}}{k!} \\
+p^{n-1}\left( -N\right) _{n-1}\,n^{-1}{\mathscr D}_{1,n}^{(j)}(x)\sum_{0%
\leq k\leq n-1}\frac{\left( n-k\right) \left( -n\right) _{k}\left( -x\right)
_{k}}{\left( -N\right) _{k}}\frac{\left( p^{-1}\right) ^{k}}{k!}.
\end{multline*}%
Thus, we have%
\begin{equation*}
\mathbb{K}_{n}^{(j)}(x)=p^{n-1}\left( -N\right) _{n-1}\sum_{0\leq k\leq
n}g_{n}^{\left( j\right) }\left( x\right) \frac{\left( -n\right) _{k}\left(
-x\right) _{k}}{\left( -N\right) _{k}}\frac{\left( p^{-1}\right) ^{k}}{k!},
\end{equation*}%
where%
\begin{equation*}
g_{n}^{\left( j\right) }\left( x\right) =-n^{-1}{\mathscr D}%
_{1,n}^{(j)}(x)\left( f_{n}^{\left( j\right) }\left( x\right) +k-1\right) ,
\end{equation*}%
with%
\begin{equation}
f_{n}^{\left( j\right) }\left( x\right) =\frac{np\left( N-n+1\right) {%
\mathscr C}_{1,n}^{(j)}(x)}{{\mathscr D}_{1,n}^{(j)}(x)}-n+1,  \label{fx}
\end{equation}%
A trivial verification shows that%
\begin{eqnarray*}
g_{n}^{\left( j\right) }\left( x\right) &=&-\frac{{\mathscr D}%
_{1,n}^{(j)}(x)\left( f_{n}^{\left( i\right) }\left( x\right) -1\right) }{n}%
\frac{\left( f_{n}^{\left( j\right) }\left( x\right) \right) _{k}}{\left(
f_{n}^{\left( j\right) }\left( x\right) -1\right) _{k}} \\
&=&-\left( p\left( N-n+1\right) {\mathscr C}_{1,n}^{(j)}(x)-{\mathscr D}%
_{1,n}^{(j)}(x)\right) \frac{\left( f_{n}^{\left( j\right) }\left( x\right)
\right) _{k}}{\left( f_{n}^{\left( j\right) }\left( x\right) -1\right) _{k}}.
\end{eqnarray*}%
Therefore%
\begin{multline*}
\mathbb{K}_{n}^{(j)}(x)=-\left( p\left( N-n+1\right) {\mathscr C}%
_{1,n}^{(j)}(x)-{\mathscr D}_{1,n}^{(j)}(x)\right) \\
\times p^{n-1}\left( -N\right) _{n-1}\sum_{0\leq k\leq n}\frac{\left(
-n\right) _{k}\left( -x\right) _{k}\left( f_{n}^{\left( j\right) }\left(
x\right) \right) _{k}}{\left( -N\right) _{k}\left( f_{n}^{\left( j\right)
}\left( x\right) -1\right) _{k}}\frac{\left( p^{-1}\right) ^{k}}{k!},
\end{multline*}%
which coincides with $\left( \text{\ref{MSPHR}}\right) $. This completes the
proof.
\end{proof}

\section{Linear difference equation of second order}

In this section, we will obtain a second order linear difference equation
that the sequence of monic Kravchuk-Sobolev type orthogonal polynomials $\{%
\mathbb{K}_{n}^{(j)}\}_{n\geq 0}$ satisfies. In order to do that, we will
find the \textit{ladder (creation and annihilation) operators}, using the
connection formula $\left( \text{\ref{SobP1}}\right) $, the three term
recurrence relation $\left( \text{\ref{ReR}}\right) $ satisfied by $%
\{K_{n}^{p,N}\}_{n\geq 0}$ and the structure relation $\left( \text{\ref%
{StruR1}}\right) $.

From $\left( \text{\ref{SobP1}}\right) $ and recurrence relation $\left( 
\text{\ref{ReR}}\right) $ we deduce the following resut%
\begin{equation}
\mathbb{K}_{n-1}^{(j)}(x)={\mathscr C}_{2,n}^{(j)}(x)K_{n}^{p,N}\left(
x\right) +{\mathscr D}_{2,n}^{(j)}(x)K_{n-1}^{p,N}\left( x\right) ,
\label{SobP2}
\end{equation}%
where%
\begin{equation*}
{\mathscr C}_{2,n}^{(j)}(x)=-\frac{{\mathscr D}_{1,n-1}^{(j)}(x)}{\beta
_{n-1}^{p,N}},
\end{equation*}%
and%
\begin{equation*}
{\mathscr D}_{2,n}^{(j)}(x)={\mathscr C}_{1,n-1}^{(j)}(x)+{\mathscr C}%
_{2,n}^{(j)}(x)\left( \alpha _{n-1}^{p,N}-x\right) .
\end{equation*}%
Applying the $\nabla $ operator to $\left( \text{\ref{SobP1}}\right) $ we
have%
\begin{eqnarray*}
\nabla \mathbb{K}_{n}^{(j)}(x) &=&K_{n}^{p,N}\left( x\right) \nabla {%
\mathscr C}_{n}^{(j)}(x)+{\mathscr C}_{n}^{(j)}(x-1)\nabla K_{n}^{p,N}\left(
x\right) \\
&&+K_{n-1}^{p,N}\left( x\right) \nabla {\mathscr D}_{n}^{(j)}(x)+{\mathscr D}%
_{n}^{(j)}(x-1)\nabla K_{n-1}^{p,N}\left( x\right) ,
\end{eqnarray*}%
Here, $\nabla $ denotes the forward operator defined by $\nabla f\left(
x\right) =f\left( x\right) -f\left( x-1\right) $. Next, multiplying the
previous expression by $x$ and using the structure relation $\left( \text{%
\ref{StruR1}}\right) $ as well as the recurrence relation $\left( \text{\ref%
{ReR}}\right) $ we deduce%
\begin{equation}
x\nabla \mathbb{K}_{n}^{(j)}(x)={\mathscr E}_{1,n}^{(j)}(x)K_{n}^{p,N}\left(
x\right) +{\mathscr F}_{1,n}^{(j)}(x)K_{n-1}^{p,N}\left( x\right) ,
\label{DQ1}
\end{equation}%
and%
\begin{equation}
x\nabla \mathbb{K}_{n-1}^{(j)}(x)={\mathscr E}_{2,n}^{(j)}(x)K_{n}^{p,N}%
\left( x\right) +{\mathscr F}_{2,n}^{(j)}(x)K_{n-1}^{p,N}\left( x\right) ,
\label{DQ2}
\end{equation}%
respectively, where%
\begin{equation*}
{\mathscr E}_{1,n}^{(j)}(x)=x\nabla {\mathscr C}_{1,n}^{(j)}(x)+n{\mathscr C}%
_{1,n}^{(j)}(x-1)-\frac{\left( n-1\right) p\left( N-n+2\right) {\mathscr D}%
_{1,n}^{(j)}(x-1)}{\beta _{n-1}^{p,N}},
\end{equation*}%
\begin{multline*}
{\mathscr F}_{1,n}^{(j)}(x)=x\nabla {\mathscr D}_{1,n}^{(j)}(x)+np\left(
N-n+1\right) {\mathscr C}_{1,n}^{(j)}(x-1) \\
+\frac{\left( n-1\right) p\left( N-n+2\right) \left( x-\alpha
_{n-1}^{p,N}\right) {\mathscr D}_{1,n}^{(j)}(x-1)}{\beta _{n-1}^{p,N}} \\
+\left( n-1\right) {\mathscr D}_{1,n}^{(j)}(x-1),
\end{multline*}%
and%
\begin{equation*}
{\mathscr E}_{2,n}^{(j)}(x)=-\frac{{\mathscr F}_{1,n-1}^{(j)}(x)}{\beta
_{n-1}^{p,N}},
\end{equation*}%
\begin{equation*}
{\mathscr F}_{2,n}^{(j)}(x)={\mathscr E}_{1,n-1}^{(j)}(x)+{\mathscr E}%
_{2,n}^{(j)}(x)\left( \alpha _{n-1}^{p,N}-x\right) .
\end{equation*}%
Clearly, from $\left( \text{\ref{SobP1}}\right) $-$\left( \text{\ref{SobP2}}%
\right) $ we deduce%
\begin{equation*}
K_{n}^{p,N}\left( x\right) =\frac{%
\begin{vmatrix}
\mathbb{K}_{n}^{(j)}(x) & \mathbb{K}_{n-1}^{(j)}(x) \\ 
{\mathscr D}_{1,n}^{(j)}(x) & {\mathscr D}_{2,n}^{(j)}(x)%
\end{vmatrix}%
}{%
\begin{vmatrix}
{\mathscr C}_{1,n}^{(j)}(x) & {\mathscr C}_{2,n}^{(j)}(x) \\ 
{\mathscr D}_{1,n}^{(j)}(x) & {\mathscr D}_{2,n}^{(j)}(x)%
\end{vmatrix}%
},\quad \mbox{and}\quad K_{n-1}^{p,N}\left( x\right) =-\frac{%
\begin{vmatrix}
\mathbb{K}_{n}^{(j)}(x) & \mathbb{K}_{n-1}^{(j)}(x) \\ 
{\mathscr C}_{1,n}^{(j)}(x) & {\mathscr C}_{2,n}^{(j)}(x)%
\end{vmatrix}%
}{%
\begin{vmatrix}
{\mathscr C}_{1,n}^{(j)}(x) & {\mathscr C}_{2,n}^{(j)}(x) \\ 
{\mathscr D}_{1,n}^{(j)}(x) & {\mathscr D}_{2,n}^{(j)}(x)%
\end{vmatrix}%
}.
\end{equation*}%
Thus, replacing the above in $\left( \text{\ref{DQ1}}\right) $-$\left( \text{%
\ref{DQ2}}\right) $ we conclude%
\begin{equation}
\Theta _{n}^{\left( j\right) }\left( x\right) \nabla \mathbb{K}%
_{n}^{(j)}(x)+\Lambda _{n}^{\left( j\right) }\left( x;2,1\right) \mathbb{K}%
_{n}^{(j)}(x)=\Lambda _{n}^{\left( j\right) }\left( x;1,1\right) \mathbb{K}%
_{n-1}^{(j)}(x).  \label{Knm1}
\end{equation}%
and%
\begin{equation*}
\Theta _{n}^{\left( j\right) }\left( x\right) \nabla \mathbb{K}%
_{n-1}^{(j)}(x)+\Lambda _{n}^{\left( j\right) }\left( x;1,2\right) \mathbb{K}%
_{n-1}^{(j)}(x)=\Lambda _{n}^{\left( j\right) }\left( x;2,2\right) \mathbb{K}%
_{n}^{(j)}(x).
\end{equation*}%
respectively, where%
\begin{equation}
\Theta _{n}^{\left( j\right) }\left( x\right) =x%
\begin{vmatrix}
{\mathscr C}_{1,n}^{(j)}(x) & {\mathscr C}_{2,n}^{(j)}(x) \\ 
{\mathscr D}_{1,n}^{(j)}(x) & {\mathscr D}_{2,n}^{(j)}(x)%
\end{vmatrix}%
,  \label{Coef1}
\end{equation}%
and%
\begin{equation}
\Lambda _{n}^{\left( j\right) }\left( x;i,k\right) =\left( -1\right) ^{k}%
\begin{vmatrix}
{\mathscr E}_{k,n}^{(j)}(x) & {\mathscr C}_{i,n}^{(j)}(x) \\ 
{\mathscr F}_{k,n}^{(j)}(x) & {\mathscr D}_{i,n}^{(j)}(x)%
\end{vmatrix}%
,\quad i=1,2,\quad k=1,2.  \label{Coef2}
\end{equation}

\begin{proposition}
Let $\left\{ \mathbb{K}_{n}^{(j)}\right\} _{n\geq 0}$ be the sequence of
monic Kravchuck-Sobolev orthogonal polynomials defined by $\left( \text{\ref%
{MSPHR}}\right) $ and let $I$ be the identity operator. Then, the ladder
(destruction and creation) operators $\mathfrak{a}$, $\mathfrak{a}^{\dagger
} $ are defined by%
\begin{equation}
\mathfrak{a}=\Theta _{n}^{\left( j\right) }\left( x\right) \nabla +\Lambda
_{n}^{\left( j\right) }\left( x;2,1\right) I,  \label{eqDO}
\end{equation}%
\begin{equation}
\mathfrak{a}^{\dagger }=\Theta _{n}^{\left( j\right) }\left( x\right) \nabla
+\Lambda _{n}^{\left( j\right) }\left( x;1,2\right) I,  \label{eqCO}
\end{equation}%
which verify%
\begin{equation}
\mathfrak{a}\left( \mathbb{K}_{n}^{(j)}(x)\right) =\Lambda _{n}^{\left(
j\right) }\left( x;1,1\right) \mathbb{K}_{n-1}^{(j)}(x),  \label{DO}
\end{equation}%
\begin{equation}
\mathfrak{a}^{\dagger }\left( \mathbb{K}_{n-1}^{(j)}(x)\right) =\Lambda
_{n}^{\left( j\right) }\left( x;2,2\right) \mathbb{K}_{n}^{(j)}(x),
\label{CO}
\end{equation}%
where $\Theta _{n}^{\left( j\right) }\left( x\right) $ and $\Lambda
_{n}^{\left( j\right) }\left( x;i,k\right) $ with $i,k=1,2$ are given in $%
\left( \text{\ref{Coef1}}\right) $-$\left( \text{\ref{Coef2}}\right) $,
respectively.
\end{proposition}

\begin{theorem}
Let $\left\{ \mathbb{K}_{n}^{(j)}\right\} _{n\geq 0}$ be the sequence of
monic polynomials orthogonal with respect to the inner product $\left( \text{%
\ref{SobIP}}\right) $. Then, the following statement holds. For all $n\geq 0$%
\begin{equation}
\mathcal{F}_{n}^{\left( j\right) }\left( x\right) \nabla ^{2}\mathbb{K}%
_{n}^{(j)}\left( x\right) +\mathcal{G}_{n}^{\left( j\right) }\left( x\right)
\nabla \mathbb{K}_{n}^{(j)}\left( x\right) +\mathcal{H}_{n}^{\left( j\right)
}\left( x\right) \mathbb{K}_{n}^{(j)}\left( x\right) =0,  \label{HEq}
\end{equation}%
where%
\begin{equation*}
\mathcal{F}_{n}^{\left( j\right) }\left( x\right) =\Theta _{n}^{\left(
j\right) }\left( x\right) \Theta _{n}^{\left( j\right) }\left( x-1\right) ,
\end{equation*}%
\begin{multline*}
\mathcal{G}_{n}^{\left( j\right) }\left( x\right) =\Theta _{n}^{\left(
j\right) }\left( x\right) (\nabla \Theta _{n}^{\left( j\right) }\left(
x\right) +\Lambda _{n}^{\left( j\right) }\left( x-1;2,1\right) +\Lambda
_{n}^{\left( j\right) }\left( x;1,2\right) ) \\
-\frac{\nabla \Lambda _{n}^{\left( j\right) }\left( x;1,1\right) (\Theta
_{n}^{\left( j\right) }\left( x\right) +\Lambda _{n}^{\left( j\right)
}\left( x;1,2\right) )\Theta _{n}^{\left( j\right) }\left( x\right) }{%
\Lambda _{n}^{\left( j\right) }\left( x;1,1\right) },
\end{multline*}%
and%
\begin{multline*}
\mathcal{H}_{n}^{\left( j\right) }\left( x\right) =\Theta _{n}^{\left(
j\right) }\left( x\right) \nabla \Lambda _{n}^{\left( j\right) }\left(
x;2,1\right) +\Lambda _{n}^{\left( j\right) }\left( x;1,2\right) \Lambda
_{n}^{\left( j\right) }\left( x;2,1\right)  \\
-\frac{\nabla \Lambda _{n}^{\left( j\right) }\left( x;1,1\right) \Lambda
_{n}^{\left( j\right) }\left( x;2,1\right) (\Theta _{n}^{\left( j\right)
}\left( x\right) +\Lambda _{n}^{\left( j\right) }\left( x;1,2\right) )}{%
\Lambda _{n}^{\left( j\right) }\left( x;1,1\right) } \\
-\Lambda _{n}^{\left( j\right) }\left( x-1;1,1\right) \Lambda _{n}^{\left(
j\right) }\left( x;2,2\right) ,
\end{multline*}%
where $\Theta _{n}^{\left( j\right) }\left( x\right) $ and $\Lambda
_{n}^{\left( j\right) }\left( x;i,k\right) $ with $i,k=1,2$ are given in $%
\left( \text{\ref{Coef1}}\right) $-$\left( \text{\ref{Coef2}}\right) $.
\end{theorem}

\begin{proof}
In fact, from \eqref{DO} we have%
\begin{equation}
\mathfrak{a}^{\dagger }\left[ \mathfrak{a}\left( \mathbb{K}%
_{n}^{(j)}(x)\right) \right] =\mathfrak{a}^{\dagger }\left[ \Lambda
_{n}^{\left( j\right) }\left( x;1,1\right) \mathbb{K}_{n-1}^{(j)}(x)\right] .
\label{eqP}
\end{equation}%
Next, applying \eqref{eqDO} to left hand member of the previous expression,
we get%
\begin{multline*}
\mathfrak{a}^{\dagger }\left[ \mathfrak{a}\left( \mathbb{K}%
_{n}^{(j)}(x)\right) \right] =\mathfrak{a}^{\dagger }\left[ \Theta
_{n}^{\left( j\right) }\left( x\right) \nabla \mathbb{K}_{n}^{(j)}(x)+%
\Lambda _{n}^{\left( j\right) }\left( x;2,1\right) \mathbb{K}_{n}^{(j)}(x)%
\right] \\
=\Theta _{n}^{\left( j\right) }\left( x\right) \nabla \left[ \Theta
_{n}^{\left( j\right) }\left( x\right) \nabla \mathbb{K}_{n}^{(j)}(x)+%
\Lambda _{n}^{\left( j\right) }\left( x;2,1\right) \mathbb{K}_{n}^{(j)}(x)%
\right] \\
+\Lambda _{n}^{\left( j\right) }\left( x;1,2\right) \Theta _{n}^{\left(
j\right) }\left( x\right) \nabla \mathbb{K}_{n}^{(j)}(x) \\
+\Lambda _{n}^{\left( j\right) }\left( x;1,2\right) \Lambda _{n}^{\left(
j\right) }\left( x;2,1\right) \mathbb{K}_{n}^{(j)}(x).
\end{multline*}%
Thus, using the property%
\begin{equation*}
\nabla \left[ f\left( x\right) g\left( x\right) \right] =\nabla f\left(
x\right) g\left( x\right) +f\left( x-1\right) \nabla g\left( x\right) ,
\end{equation*}%
we deduce%
\begin{multline}
\mathfrak{a}^{\dagger }\left[ \mathfrak{a}\left( \mathbb{K}%
_{n}^{(j)}(x)\right) \right] =\Theta _{n}^{\left( j\right) }\left( x\right)
\Theta _{n}^{\left( j\right) }\left( x-1\right) \nabla ^{2}\mathbb{K}%
_{n}^{(j)}(x) \\
+\Theta _{n}^{\left( j\right) }\left( x\right) (\nabla \Theta _{n}^{\left(
j\right) }\left( x\right) +\Lambda _{n}^{\left( j\right) }\left(
x-1;2,1\right) +\Lambda _{n}^{\left( j\right) }\left( x;1,2\right) )\nabla 
\mathbb{K}_{n}^{(j)}(x) \\
+(\Theta _{n}^{\left( j\right) }\left( x\right) \nabla \Lambda _{n}^{\left(
j\right) }\left( x;2,1\right) +\Lambda _{n}^{\left( j\right) }\left(
x;1,2\right) \Lambda _{n}^{\left( j\right) }\left( x;2,1\right) )\mathbb{K}%
_{n}^{(j)}(x).  \label{iqI}
\end{multline}%
On the other hand, having into account \eqref{eqCO}, \eqref{CO} and %
\eqref{Knm1} we deduce%
\begin{multline}
\mathfrak{a}^{\dagger }\left[ \Lambda _{n}^{\left( j\right) }\left(
x;1,1\right) \mathbb{K}_{n-1}^{(j)}(x)\right] =\Big[\Lambda _{n}^{\left(
j\right) }\left( x-1;1,1\right) \Lambda _{n}^{\left( j\right) }\left(
x;2,2\right) \\
+\frac{\nabla \Lambda _{n}^{\left( j\right) }\left( x;1,1\right) \Lambda
_{n}^{\left( j\right) }\left( x;2,1\right) (\Theta _{n}^{\left( j\right)
}\left( x\right) +\Lambda _{n}^{\left( j\right) }\left( x;1,2\right) )}{%
\Lambda _{n}^{\left( j\right) }\left( x;1,1\right) }\Big]\mathbb{K}%
_{n}^{(j)}(x) \\
+\frac{\nabla \Lambda _{n}^{\left( j\right) }\left( x;1,1\right) (\Theta
_{n}^{\left( j\right) }\left( x\right) +\Lambda _{n}^{\left( j\right)
}\left( x;1,2\right) )\Theta _{n}^{\left( j\right) }\left( x\right) }{%
\Lambda _{n}^{\left( j\right) }\left( x;1,1\right) }\nabla \mathbb{K}%
_{n}^{(j)}(x).  \label{iqII}
\end{multline}%
Finally, equaling \eqref{iqI} and \eqref{iqII} we arrived to the desired
result.
\end{proof}

\begin{theorem}
Let $\left\{ \mathbb{K}_{n}^{(j)}\right\} _{n\geq 0}$ be the sequence of
monic polynomials orthogonal with respect to the inner product $\left( \text{%
\ref{SobIP}}\right) $. Then, the following statement holds. For all $n\geq 0$%
\begin{equation}
\mathcal{\tilde{F}}_{n}^{\left( j\right) }\left( x\right) \Delta \nabla 
\mathbb{K}_{n}^{(j)}\left( x\right) +\mathcal{\tilde{G}}_{n}^{\left(
j\right) }\left( x\right) \Delta \mathbb{K}_{n}^{(j)}\left( x\right) +%
\mathcal{\tilde{H}}_{n}^{\left( j\right) }\left( x\right) \mathbb{K}%
_{n}^{(j)}\left( x\right) =0,  \label{HypEqq}
\end{equation}%
where%
\begin{equation*}
\mathcal{\tilde{F}}_{n}^{\left( j\right) }\left( x\right) =\mathcal{F}%
_{n}^{\left( j\right) }\left( x+1\right) ,\quad \mbox{and}\quad \mathcal{%
\tilde{H}}_{n}^{\left( j\right) }\left( x\right) =\mathcal{H}_{n}^{\left(
j\right) }\left( x+1\right) .
\end{equation*}%
and%
\begin{equation*}
\mathcal{\tilde{G}}_{n}^{\left( j\right) }\left( x\right) =\mathcal{G}%
_{n}^{\left( j\right) }\left( x+1\right) +\mathcal{H}_{n}^{\left( j\right)
}\left( x+1\right) .
\end{equation*}
\end{theorem}

\begin{proof}
In fact, applying the property $\nabla f\left( x\right) =\Delta f\left(
x-1\right) $ to \eqref{HEq} we have%
\begin{equation*}
\mathcal{F}_{n}^{\left( j\right) }\left( x\right) \nabla \Delta \mathbb{K}%
_{n}^{(j)}\left( x-1\right) +\mathcal{G}_{n}^{\left( j\right) }\left(
x\right) \Delta \mathbb{K}_{n}^{(j)}\left( x-1\right) +\mathcal{H}%
_{n}^{\left( j\right) }\left( x\right) \mathbb{K}_{n}^{(j)}\left( x\right)
=0.
\end{equation*}%
Then, replacing $x$ by $x+1$, we deduce%
\begin{multline*}
\mathcal{F}_{n}^{\left( j\right) }\left( x+1\right) \Delta \nabla \mathbb{K}%
_{n}^{(j)}\left( x\right) +\mathcal{G}_{n}^{\left( j\right) }\left(
x+1\right) \Delta \mathbb{K}_{n}^{(j)}\left( x\right)  \\
+\mathcal{H}_{n}^{\left( j\right) }\left( x+1\right) \mathbb{K}%
_{n}^{(j)}\left( x+1\right) -\mathcal{H}_{n}^{\left( j\right) }\left(
x+1\right) \mathbb{K}_{n}^{(j)}\left( x\right)  \\
+\mathcal{H}_{n}^{\left( j\right) }\left( x+1\right) \mathbb{K}%
_{n}^{(j)}\left( x\right) =0.
\end{multline*}%
which coincides with \eqref{HypEqq}.
\end{proof}

\bibliographystyle{plain}
%11%%\bibliography{Refs}

\def\cprime{$'$}

%\end{thebibliography}

\end{document}